\documentclass[12pt,reqno]{amsart}
\usepackage{latexsym,amsmath,amssymb, amsthm, mathscinet}
\usepackage{cases, verbatim}
\usepackage{graphicx}
\usepackage{bmpsize}
\usepackage{float}

\usepackage{amsfonts,amsmath,amsthm, amssymb}
\usepackage{cases}
\usepackage[usenames]{color}
\usepackage{enumerate}
\usepackage{bm}

\theoremstyle{plain}
\newtheorem{thm}{Theorem}[section]

\newtheorem{lem}[thm]{Lemma}
\newtheorem{cor}[thm]{Corollary}
\newtheorem{prop}[thm]{Proposition}
\theoremstyle{remark}
\newtheorem{rem}{\bf{Remark}}
\numberwithin{equation}{section}

\newcommand{\R}{\mathbb{R}}

\newcommand{\T}{\mathbb{T}}
\newcommand{\Z}{\mathbb{Z}}

\newcommand{\AC}{{\rm AC\,}}

\newcommand{\al}{\alpha}

\newcommand{\ep}{\varepsilon}

\newcommand{\lam}{\lambda}

\newcommand{\Div}{{\rm div}\,}

\newcommand{\tr}{{\rm tr}\,}

\begin{document}
\baselineskip=15pt

\title[Vanishing discount and viscosity selection problems]
{Vanishing discount and viscosity selection problems for mechanical Hamilton--Jacobi equations}

\author{Hiroyoshi Mitake}
\address[H. Mitake]{
	Department of Mathematics, 
	Faculty of Science and Engineering, 
	Waseda University, 
	3-4-1 Okubo, Shinjuku-ku, Tokyo, 169-8555, Japan}
\email{mitake@waseda.jp}

\author{Panrui Ni}
\address[P. Ni]{
Department 1: Department of Mathematics, Faculty of Science and Engineering, Waseda University, 3-4-1 Okubo, Shinjuku-ku, Tokyo, 169-8555, Japan; Department 2: Shanghai Center for Mathematical Sciences, Fudan University, Shanghai 200438, China}
\email{panruini@gmail.com}

\author{Hung V. Tran}
\address[H. V. Tran]{
Department of Mathematics, University of Wisconsin Madison, Van Vleck Hall, 480 Lincoln Drive, Madison, Wisconsin 53706, USA}
\email{hung@math.wisc.edu}

\makeatletter
\@namedef{subjclassname@2020}{\textup{2020} Mathematics Subject Classification}
\makeatother

\date{\today}

\begin{abstract}
We study simultaneous vanishing discount and viscosity limits for mechanical Hamilton--Jacobi equations on the $n$-dimensional torus. 
In the critical regime, we explicitly identify the selected limit in terms of the critical Ma\~n\'e potential and the quadratic behavior of the potential near its
wells. 
We also obtain quantitative and sharp convergence estimates in the subcritical regime. 
In the supercritical regime, we prove convergence when
the local harmonic ground-state energy of the potential has a unique minimizer among the wells. 
Finally, we establish a concentration phenomenon of the associated Gibbs measures.
\end{abstract}

\date{\today}

\maketitle



\section{Introduction}
In this paper, we are interested in the vanishing discount and viscosity selection problems for a mechanical Hamilton--Jacobi equation.

Let $n\geq 1$, and let $\T^n=\R^n/\Z^n$ be the flat $n$-dimensional torus.
Our main focus is the asymptotic behavior of the solution $v^\ep$ of the critical-scaling vanishing discount and viscosity equation
\begin{equation}\label{eq:discount-viscous}
    \ep v^\ep-\ep\Delta v^\ep+|Dv^\ep|^2=V
    \qquad\text{in }\T^n
\end{equation}
as $\ep \to 0+$.
Here, \eqref{eq:discount-viscous} is of critical scaling, as the discount and the viscosity terms are both of order $\ep$ and they compete with each other in the limit.
We \textit{always} assume that the potential energy $V\in C^2(\T^n)$ satisfies
\[
\begin{cases}
V\geq 0 \qquad \text{on} \ \T^n,\\  
\{V=0\}=\{x_1,\ldots,x_m\},\\
D^2V(x_i)>0 \qquad (1\leq i\leq m).
\end{cases}   
\]
Since equation \eqref{eq:discount-viscous} is uniformly elliptic for every fixed $\ep>0$, its viscosity solution $v^\ep$ is classical. 
At each well $x_i$ of $V$ for $1\le i\le m$, define the positive definite symmetric matrix
\begin{equation}\label{eq:Ai-kappai}
    A_i:=\left(\frac{D^2V(x_i)}{2}\right)^{1/2},
    \qquad
    \kappa_i:=\tr A_i.
\end{equation}
We call $\kappa_i$ {\it the local harmonic ground-state energy} of $V$ at the well $x_i$.
Thus, in local Euclidean coordinates centered at $x_i$,
\begin{equation}\label{eq:V-quadratic-expansion}
    V(x_i+y)=y\cdot A_i^2y+o(|y|^2)
    \qquad\text{as }y\to 0.
\end{equation}

Equation \eqref{eq:discount-viscous} has a direct interpretation as the dynamic programming equation of an infinite-horizon discounted stochastic control problem.
The Lagrangian associated with the Hamiltonian $H(x,p)=|p|^2-V(x)$ is
\[
    L(x,q)=\frac{|q|^2}{4}+V(x).
\]
Then,
\begin{equation}\label{eq:control-representation-intro}
    v^\ep(x)
    =
    \inf_{\beta}
    \mathbb E_x\left[
        \int_0^\infty e^{-\ep t}
        L(X_t,\beta_t)\,dt
    \right],
    \qquad
    dX_t=\beta_t\,dt+\sqrt{2\ep}\,dW_t,
\end{equation}
where the infimum is taken over admissible controls.
See \cite{FS} for instance.
The representation formula \eqref{eq:control-representation-intro} makes the critical nature of the scaling transparent.
The parameter $\ep$ is simultaneously the discount rate and the diffusivity of the controlled process.  
Thus, the effective discounted time horizon is of order $\ep^{-1}$, whereas the noise has variance of order $\ep$ per unit time.  
Consequently, an $O(\ep)$ fluctuation cost per unit time accumulates to an $O(1)$ cost over the discounted horizon.

We also denote by $c(\ep)$ the unique constant for which the viscous ergodic problem
\begin{equation}\label{eq:ergodic}
    -\ep\Delta \chi^\ep+|D\chi^\ep|^2=V+c(\ep)\qquad\text{in }\T^n
\end{equation}
admits a periodic solution. The solution is unique up to additive constants. We will use the asymptotic formula, for $\kappa_*:=\min_{1\le i\le m}\kappa_i$,
\begin{equation}\label{eq:c-asymptotic-intro}
    c(\ep)=-\ep\kappa_*+o(\ep).
\end{equation}
See Appendix~\ref{app:known-results}.
After the Hopf--Cole transform, $-c(\ep)$ is the principal eigenvalue of the periodic Schr\"odinger operator $-\ep^2\Delta+V$. 
Near $x_i$, the operator is $-\ep^2\Delta+y\cdot A_i^2y+o(|y|^2)$, while the ground-state energy of $-\ep^2\Delta+y\cdot A_i^2y$ in $\R^n$ is $\ep\tr A_i=\ep\kappa_i$. Thus $\ep\kappa_i$ is the first-order ground-state energy associated with the quadratic approximation of $V$ at $x_i$.

The critical Ma\~n\'e potential is
\[
    d(y,x):=
    \inf_{T>0}\ 
    \inf_{\substack{\gamma\in \AC([0,T];\T^n)\\
                     \gamma(0)=y,\ \gamma(T)=x}}
    \int_0^T
    \left(\frac{|\dot\gamma(t)|^2}{4}+V(\gamma(t))\right)dt.
\]
Set
\begin{equation}\label{eq:v-star}
    v^*(x):=\min_{1\leq i\leq m}
    \bigl\{\kappa_i+d(x_i,x)\bigr\}.
\end{equation}
Note that $v^\ast$ is a viscosity solution to the cell problem (or the limit problem) 
\begin{equation}\label{eq:cell}
    |Dv|^2=V\quad\text{in} \ \T^n. 
\end{equation}

Here is the main result of our paper.
\begin{thm}\label{thm:critical-case}
Let $v^\ep$ be the solution of \eqref{eq:discount-viscous}, and $v^*$ be as in \eqref{eq:v-star}.
Then, uniformly for $x\in \T^n$,
\[
    \lim_{\ep\to 0}v^\ep(x)= v^*(x).
\]
\end{thm}

\begin{rem}\label{rem:1}
In Theorem \ref{thm:critical-case}, the form of $v^*$ can be anticipated directly from the control representation.  
As noted, near a well $x_i$, the quadratic approximation is $V(x_i+y)\simeq y\cdot A_i^2y$.
The stabilizing feedback $\beta=-2Dv^\ep \simeq -2A_i y$ produces, at leading order, the
Ornstein--Uhlenbeck process
\[
    dY_t=-2A_iY_t\,dt+\sqrt{2\ep}\,dW_t.
\]
Its invariant covariance is $(\ep/2)A_i^{-1}$, and hence its mean running
cost is
\[
    \mathbb E\left[
        \frac{|\beta|^2}{4}+Y_t\cdot A_i^2Y_t
    \right]
    =
    \ep\tr A_i
    =
    \ep\kappa_i.
\]
Discounting at rate $\ep$ converts this into the order-one cost
\[
    \int_0^\infty e^{-\ep t}\ep\kappa_i\,dt=\kappa_i.
\]
On the other hand, the leading-order transient cost of reaching $x_i$ from $x$ is the critical Ma\~n\'e action $d(x_i,x)$.  
Thus, one expects the controller to choose the well minimizing
\[
    \kappa_i+d(x_i,x),
\]
which is precisely the function $v^*$ in \eqref{eq:v-star}.
See Appendix \ref{appendix: another proof} for some further information.
\end{rem}

The same argument, with only the coefficient of the discount term changed, gives the following corollary.

\begin{cor}\label{cor:critical}
    For each fixed $\eta>0$, consider the modified critical-scaling vanishing discount and viscosity equation
\[
    \ep \eta v^{\eta,\ep}-\ep\Delta v^{\eta,\ep}+|Dv^{\eta,\ep}|^2=V
    \qquad\text{in }\T^n.
\]
Then, uniformly for $x\in \T^n$,
\[
    \lim_{\ep\to 0}v^{\eta,\ep}(x)= v^\eta(x)=\min_{1\leq i\leq m}
    \left\{\frac{\kappa_i}{\eta}+d(x_i,x)\right\} .
\]
\end{cor}

\begin{rem}\label{rem:critical case}
It is important to note that in the critical scaling case, the discount and the viscosity terms are both of order $\ep$, and they compete with each other in the limit.
We have the convergence results in Theorem \ref{thm:critical-case} and Corollary \ref{cor:critical} because the discount term stabilizes the process; see Remark \ref{rem:1}.
This is consistent with the fact that the vanishing discount process with convex Hamiltonians selects a unique limit (see, for instance, \cite{DFIZ} for the first-order equations, with earlier results in \cite{Gomes, ISM}, and \cite{MitakeTran, IMT} for the second-order equations).
The stabilization given by the discount term also plays an important role in Theorem \ref{thm: sum up} below.
However, with only viscosity terms, the vanishing viscosity process might fail to select a unique limit; see \cite{LTY}.

\end{rem}

It is also natural to study the vanishing discount and viscosity problems when they have different scalings.
A generalized form of \eqref{eq:discount-viscous} is
\begin{equation}\label{eq:alpha}
    \ep^\alpha w^\ep -\ep \Delta w^\ep + |Dw^\ep|^2= V \qquad \text{in } \T^n.
\end{equation}
Here, $\alpha>0$ is a given parameter.
Theorem \ref{thm:critical-case} covers the critical case where $\alpha=1$.
In the general Tonelli setting, convergence was proved in \cite{WZ} under the assumption that the viscosity coefficient $\ep(\lambda)$ satisfies $\ep(\lambda)/\lambda\to0$ as the discount coefficient $\lambda\to0^+$, while the regime $\ep(\lambda)=O(\lambda)$ was left open. Thus, the critical case $\alpha=1$ considered here lies precisely at this borderline. In the present mechanical setting we obtain convergence and identify the selected limit explicitly.
Another motivation for considering different relative scalings of the discount and viscosity terms comes from the broader question of how different approximation procedures select among the viscosity solutions of the cell problem. Selection principles for viscosity solutions and Mather measures were studied in \cite{Gomes}.
In \cite{NYZ}, different elementary solutions were obtained by varying the discount coefficient on the static classes, while \cite{CZ} showed that a discounted approximation combined with a small potential perturbation can select any prescribed solution of the ergodic problem. 
In the present problem, we also see different limits arise from the relative size of the discount and viscosity terms.

\medskip

If $0<\alpha<1$, we are in the subcritical case where the discount term dominates the diffusion term. 
The convergence of $\{w^\ep\}$ in this regime follows from \cite{WZ}. 
Set
\begin{equation}\label{func:u0}
u^0(x):=\min_{1\le i\le m}d(x_i,x).    
\end{equation}
    
Quantitative convergence rates for the first-order vanishing discount problem were studied in \cite{MitakeSoga} under additional dynamical assumptions, and sharp rates for hyperbolic Aubry sets were obtained in \cite{NiRate}.
In the present mechanical setting, we obtain quantitative estimates for $w^\ep\to u^0$, including a sharp rate when $1/2\le\alpha<1$; see Theorem~\ref{thm: sum up} (i).

\medskip

If $\alpha>1$, we are in the supercritical case where the diffusion term dominates the discount term.
The pure vanishing viscosity problem, which corresponds formally to the endpoint $\alpha=\infty$, was studied in \cite{AIPSM, JKM, LTY}.
Under the additional assumption \eqref{eq:unique-j}, which was the key condition imposed in \cite{AIPSM, JKM}, we prove that for every finite $\alpha>1$ the normalized solutions converge uniformly to $d(x_j,\cdot)$.
We sum up the convergence results in the following theorem.

\begin{thm}\label{thm: sum up}
We have the following conclusions.
\begin{enumerate}
    \item[{\rm(i)}] If $\alpha \in (0,1)$, then uniformly for $x\in \T^n$,
\[
    \lim_{\ep\to 0}\left(w^\ep(x)+\frac{c(\ep)}{\ep^\alpha}\right)= \lim_{\ep\to 0}w^\ep(x)=u^0(x)=\min_{1\le i\le m} d(x_i,x),
\]
and there exists $C>0$ such that
\begin{equation}\label{eq:subcritical-global-rate}
    \|w^\ep-u^0\|_{L^\infty(\T^n)}
    \le C\bigl(\ep^\alpha+\ep^{1-\alpha}\bigr)
\end{equation}
for all sufficiently small $\ep>0$.
Moreover, when $1/2\le\alpha<1$, there exist $c,C>0$ such that, for $\ep>0$ sufficiently small,
\[
    c\ep^{1-\alpha}\le \|w^\ep-u^0\|_{L^\infty(\T^n)}\le C\ep^{1-\alpha}.
\]

\item[{\rm(ii)}] If $\alpha>1$ and there exists a unique $j\in \{1,\ldots,m\}$ such that
\begin{equation}\label{eq:unique-j}
0<\kappa_j < \kappa_i \qquad \text{for all } i \in \{1,\ldots,m\}\setminus \{j\},
\end{equation}
then uniformly for $x\in \T^n$,
\[
    \lim_{\ep\to 0}\left(w^\ep(x)+\frac{c(\ep)}{\ep^\alpha}\right)= d(x_j,x).
\]

\end{enumerate}
\end{thm}

\begin{rem}\label{rem:two-parameter}
The power-law relation between the discount and viscosity parameters is not essential.
More generally, for each $\ep>0$, fix $\lambda_\ep>0$ such that $\lim_{\ep \to 0}\lambda_\ep=0$.
Let $u^{\lambda_\ep,\ep}$ solve
\[
    \lambda_\ep u^{\lambda_\ep,\ep}-\ep\Delta u^{\lambda_\ep,\ep}
    +|Du^{\lambda_\ep,\ep}|^2=V
    \qquad\text{in }\T^n.
\]
The same arguments give the following three regimes:
\[
\begin{cases}
 u^{\lambda_\ep,\ep}\to u^0
 \quad& \displaystyle \ \text{if} \  \frac{\lambda_\ep}{\ep}\to\infty,\\[2mm]
 u^{\lambda_\ep,\ep}\to
 \displaystyle\min_{1\le i\le m}\left\{\frac{\kappa_i}{\eta}+d(x_i,\cdot)\right\}
 \quad& \displaystyle \ \text{if} \  \frac{\lambda_\ep}{\ep}\to\eta\in(0,\infty),\\[3mm]
 u^{\lambda_\ep,\ep}+\dfrac{c(\ep)}{\lambda_\ep}\to d(x_j,\cdot)
 \quad& \displaystyle \ \text{if} \ \frac{\lambda_\ep}{\ep}\to0,
\end{cases}
\]
where the last conclusion is under \eqref{eq:unique-j}.
\end{rem}

As noted, in the supercritical case $\alpha>1$, we need to assume \eqref{eq:unique-j} to have the convergence of $w^\ep(x)+\frac{c(\ep)}{\ep^\alpha}$.
Without this condition, \cite{LTY} constructs a one-dimensional nonconvergence example for the pure vanishing viscosity process, which corresponds formally to $\alpha=\infty$ in \eqref{eq:alpha}.
We can adapt the method in \cite{LTY} to obtain a nonconvergence result in this supercritical setting as follows.

\begin{thm}\label{thm:nonconvergence}
Let $n=1$ and $\alpha>1$.
There exists a potential $V\in C^2(\T)$, with exactly two nondegenerate
zeroes, for which the normalized solutions
\[
 z^\ep:=w^\ep+\frac{c(\ep)}{\ep^\alpha}
\]
do not converge in $C(\T)$ as $\ep\to0$.
\end{thm}

Finally, under \eqref{eq:unique-j}, we have the convergence of the corresponding Gibbs measures, which shows a concentration of measure phenomenon.

\begin{thm}\label{thm:measure}
Fix $\alpha>0$.
Assume that \eqref{eq:unique-j} holds.
For $\ep\in (0,1)$, define
\begin{equation}\label{eq:rho}
    \rho^\ep(x):=\frac{e^{-w^\ep(x)/\ep}}{\displaystyle\int_{\T^n}e^{-w^\ep(y)/\ep}\,dy}.
\end{equation}
Then, in the sense of measures as $\ep \to 0$,
\begin{equation}\label{eq:rho-concentration-intro}
    \rho^\ep(x)\,dx\rightharpoonup\delta_{x_j}.
\end{equation}
\end{thm}

Because of \eqref{eq:unique-j}, it is clear that the well at $x_j$ is the unique preferred well with the lowest ground-state energy $\ep \kappa_j=\ep \kappa_*$.
By essentially using the uniform convergence of $w^\ep$ with a scale translation given by Theorem \ref{thm: sum up}, we can establish \eqref{eq:rho-concentration-intro}. 
We also notice that by a Hopf--Cole transformation $\varphi^\ep:=e^{-\frac{w^\ep}{\ep}}$, \eqref{eq:alpha} turns out to be a logarithmic Schr\"odinger equation 
\begin{equation}\label{eq:log}
-\ep^2\Delta\varphi^\ep+V\varphi^\ep=-\ep^{1+\alpha}\varphi^\ep\log\varphi^\ep\qquad\text{in} \ \T^n. 
\end{equation}
Therefore, \eqref{eq:rho-concentration-intro} shows a concentration phenomenon of a scaling solution $\frac{\varphi^\ep}{\int_{\T^n}\varphi^\ep(y)\,dy}$ of \eqref{eq:log}. 
For related results on semiclassical measures associated with Schrödinger operators, see \cite{Yu}.
 



\subsection*{Organization of the paper}
The proof of Theorem \ref{thm:critical-case} is given in Section~\ref{sec:main}.
In Sections \ref{sec:subcritical} and \ref{sec:supercritical}, we prove Theorem \ref{thm: sum up}.
We give a sketch of the proof of Theorem \ref{thm:nonconvergence} in Section \ref{sec:supercritical}.
The proof of Theorem \ref{thm:measure} is given in Section~\ref{sec:cole-hopf}.
In Appendix \ref{app:known-results}, we discuss some known results used in this paper.
In Appendix \ref{appendix: another proof}, we give another proof of Proposition~\ref{prop:upper-bound} by using stochastic control.

\section{Proof of Theorem \ref{thm:critical-case}}\label{sec:main}

By the standard uniform Lipschitz estimate for coercive viscous Hamilton--Jacobi equations (e.g., \cite{LMT}), there is $C_0>0$, independent of $\ep\in(0,1)$, such that 
\[
\|Dv^\ep\|_{L^\infty(\T^n)}\le C_0.
\]

\subsection{Upper bound}
\begin{prop}\label{prop:upper-bound}
We have
\[
    \limsup_{\ep\to 0}v^\ep(x)\leq v^*(x)
    \qquad\text{for every }x\in\T^n.
\]
\end{prop}

In fact, to prove Proposition \ref{prop:upper-bound}, we only need to prove the following.

\begin{lem}\label{lem:well-bound}
For every $i\in\{1,\ldots,m\}$,
\begin{equation}\label{eq:well-limsup}
    \limsup_{\ep\to 0}v^\ep(x_i)\leq \kappa_i.
\end{equation}
\end{lem}

\begin{proof}
Fix $i\in\{1,\ldots,m\}$ and a small radius $r>0$. 
We work in $B(x_i,r)$ and write points as $x_i+y$.
For notational simplicity, set
\[
    A:=A_i,
    \qquad
    \kappa:=\kappa_i, 
\]
where $A_i$ and $\kappa_i$ are given by \eqref{eq:Ai-kappai}. 

Fix $\delta>0$ and define
\[
    A_\delta:=A+\delta I.
\]
The matrix
\[
    A_\delta^2-A^2=2\delta A+\delta^2I
\]
is positive definite. 
Let
\[
    \mu_\delta:=\lambda_{\min}(A_\delta^2-A^2)>0,
\]
where we denote $\lambda_{\min}(M)$ as the smallest eigenvalue of $M$ for a given symmetric matrix $M$.
By \eqref{eq:V-quadratic-expansion}, after choosing $r>0$ sufficiently small, we have, for $y\in B(0,r)$,
\begin{equation}\label{eq:strict-quadratic-majorant}
    y\cdot A_\delta^2y-V(x_i+y)
    \geq \frac{\mu_\delta}{2}|y|^2.
\end{equation}

For $\ep>0$, define
\[
    P_{\ep,\delta}
    :=\left(A_\delta^2+\frac{\ep^2}{16}I\right)^{1/2}
      -\frac{\ep}{4}I,
\]
which is a positive definite symmetric matrix.
Moreover,
\begin{equation}\label{eq:matrix-riccati}
    P_{\ep,\delta}^2+\frac{\ep}{2}P_{\ep,\delta}
    =A_\delta^2,
\end{equation}
and
\begin{equation}\label{eq:P-limit}
    P_{\ep,\delta}\longrightarrow A_\delta
    \qquad\text{as }\ep\to 0.
\end{equation}

Choose a nondecreasing function $\zeta\in C^\infty([0,\infty))$ such that
\[
    0\leq\zeta\leq 1,
    \qquad
    \zeta=0\text{ on }[0,1/3],
    \qquad
    \zeta=1\text{ on }[2/3,\infty).
\]
Recall that $\|Dv^\ep\|_{L^\infty(\T^n)}\le C_0$ uniformly in $\ep$.
Let $M=C_0r+1$ and
\[
    \psi(y):=M\zeta\left(\frac{|y|}{r}\right).
\]
It is clear that $\psi$ is smooth. 
On the annulus $r/3<|y|<2r/3$ where $D\psi$ may be nonzero,
\begin{equation}\label{eq:cross-positive}
    (P_{\ep,\delta}y)\cdot D\psi(y)
    =\frac{M}{r}\zeta'\left(\frac{|y|}{r}\right)
      \frac{y\cdot P_{\ep,\delta}y}{|y|}
    \geq 0.
\end{equation}

We claim that, for all sufficiently small $\ep$, one has
\begin{equation}\label{eq:finite-eps-well-bound}
    v^\ep(x_i)\leq \tr P_{\ep,\delta}.
\end{equation}
Assume by contradiction that
\[
    v^\ep(x_i)>\tr P_{\ep,\delta}.
\]
Define, for $|y|\leq r$,
\[
    W(y):=v^\ep(x_i)
           +\frac12y\cdot P_{\ep,\delta}y
           +\psi(y).
\]
We verify that $W$ is a supersolution of \eqref{eq:discount-viscous} in $B(x_i,r)$. 
By \eqref{eq:matrix-riccati},
\begin{align}
&\ep W-\ep\Delta W+|DW|^2-V(x_i+y)\notag\\
=\ &
  \ep\bigl(v^\ep(x_i)-\tr P_{\ep,\delta}\bigr)
  +y\cdot A_\delta^2y-V(x_i+y)\notag\\
&\qquad
  +\ep\psi-\ep\Delta\psi
  +2(P_{\ep,\delta}y)\cdot D\psi
  +|D\psi|^2
  \label{eq:W-residual}
\\
\ge\,& 
\Big(y\cdot A_\delta^2y-V(x_i+y)\Big)+\ep\psi-\ep\Delta\psi, 
\notag
\end{align}
where we use \eqref{eq:finite-eps-well-bound} and \eqref{eq:cross-positive} in the last inequality. 
We notice that because of the criticality of the competition between the discount and viscous terms, we can control the first term in \eqref{eq:W-residual}. 

On $|y|\leq r/3$ and on $2r/3\leq |y|\leq r$, the function $\psi$ is constant, so $\Delta\psi=0$. On the remaining annulus, $r/3< |y|< 2r/3$, estimate \eqref{eq:strict-quadratic-majorant} gives
\[
    y\cdot A_\delta^2y-V(x_i+y)
    \geq \frac{\mu_\delta r^2}{18}.
\]
Therefore, if
\[
    0<\ep\leq
    \frac{\mu_\delta r^2}
         {18\max\{1,\|\Delta\psi\|_{L^\infty(B(0,r))}\}},
\]
then the right-hand side of \eqref{eq:W-residual} is nonnegative throughout
$B(0,r)$.

Thus, $W$ is a supersolution in $B(x_i,r)$.
On $\partial B(x_i,r)$, the choice of $M$ gives
\[
    v^\ep(x_i+y)
    \leq v^\ep(x_i)+C_0r
    <v^\ep(x_i)+M
    \leq W(y).
\]
The comparison principle in $B(x_i,r)$ therefore yields, for $|y|\leq r$,
\begin{equation}\label{eq:v-below-W}
    v^\ep(x_i+y)\leq W(y).
\end{equation}

At $y=0$, equality holds in \eqref{eq:v-below-W}. Hence $v^\ep-W$ has a
local maximum at $0$. Since $\psi$ is identically zero near $0$,
\[
    DW(0)=0,
    \qquad
    \Delta W(0)=\tr P_{\ep,\delta}.
\]
It follows that
\[
    Dv^\ep(x_i)=0,
    \qquad
    \Delta v^\ep(x_i)\leq \tr P_{\ep,\delta}.
\]
Evaluating \eqref{eq:discount-viscous} at the well $x_i$, where
$V(x_i)=0$, gives
\[
    0
    =\ep v^\ep(x_i)-\ep\Delta v^\ep(x_i)
    \geq
    \ep\bigl(v^\ep(x_i)-\tr P_{\ep,\delta}\bigr)>0,
\]
which is a contradiction. 
Thus, \eqref{eq:finite-eps-well-bound} holds.

Let $\ep\to 0$ and use \eqref{eq:P-limit} to get
\[
    \limsup_{\ep\to 0}v^\ep(x_i)
    \leq \tr A_\delta
    =\tr A+n\delta
    =\kappa_i+n\delta.
\]
Finally, let $\delta\downarrow 0$ to obtain
\eqref{eq:well-limsup}.
\end{proof}

\begin{proof}[Proof of Proposition~\ref{prop:upper-bound}]
By the uniform Lipschitz bound and Lemma~\ref{lem:well-bound}, every sequence $\{\ep_k\}\to0$ has a subsequence $\{\ep_{k_j}\}\to0$ along which $v^{\ep_{k_j}}\to v$ uniformly. The stability of viscosity solutions gives $|Dv|^2=V$. Hence, for every $i$, the domination property gives $v(x)-v(x_i)\le d(x_i,x)$, and Lemma~\ref{lem:well-bound} yields $v(x)\le\kappa_i+d(x_i,x)$. Taking the minimum over $i$ proves the proposition.
\end{proof}

For a different proof of Proposition~\ref{prop:upper-bound} by stochastic control, see Appendix \ref{appendix: another proof}.

\subsection{Lower bound}
Set $z^\ep(x):=v^\ep(x)+\frac{c(\ep)}{\ep}$ for $x\in \T^n$, and then
\begin{equation}\label{eq:z-ep}
    \ep z^\ep - \ep \Delta z^\ep + |Dz^\ep|^2=V + c(\ep) \qquad \text{in } \T^n.
\end{equation}
Without loss of generality, we assume $\kappa_1=\min_{1\le i\le m} \kappa_i$.
Then, by Theorem \ref{thm:asymptotic-c-ep},
\[
\lim_{\ep \to 0} \frac{c(\ep)}{\ep}=-\min_{1\le i \le m} \kappa_i=-\kappa_1.
\]
To prove the lower bound, we have the following lemma.
\begin{lem}\label{lem:key-lower-bound}
Assume that $z^\ep \to z$ uniformly on $\T^n$ by passing to a subsequence if needed.
If $z$ has a local min at $x_j$ for some $j\in\{1,\ldots,m\}$, then
\[
z(x_j)=\kappa_j-\kappa_1.
\]
\end{lem}

\begin{proof}
    Of course, $z$ solves the cell problem
    \[
    |Dz|^2=V \qquad \text{in } \T^n.
    \]
    By Lemma \ref{lem:cell problem at local min}, $z$ is $C^2$ in a neighborhood of $x_j$ and
    \[
    D^2z(x_j)=A_j, \qquad \Delta z(x_j)=\tr A_j=\kappa_j.
    \]
Thus, $z$ has a strict local minimum at $x_j$.

For $\theta>1$, $(\theta-1)z$ has a strict local minimum at $x_j$. Hence $\theta z^\ep-z$ has a local minimum at some $x_\ep\to x_j$, and
\[
    \theta Dz^\ep(x_\ep)=Dz(x_\ep),\qquad \theta\Delta z^\ep(x_\ep)\ge\Delta z(x_\ep).
\]
Using \eqref{eq:z-ep} and $|Dz|^2=V$,
\begin{align*}
    c(\ep)
    &\le \ep z^\ep(x_\ep)-\frac{\ep}{\theta}\Delta z(x_\ep)
      +\left(\frac1{\theta^2}-1\right)V(x_\ep)\\
    &\le \ep z^\ep(x_\ep)-\frac{\ep}{\theta}\Delta z(x_\ep).
\end{align*}
Dividing by $\ep$ and letting $\ep\to0$ gives $-\kappa_1\le z(x_j)-\kappa_j/\theta$. Letting $\theta\downarrow1$ yields $z(x_j)\ge\kappa_j-\kappa_1$.

For $0<\theta<1$, $(\theta-1)z$ has a strict local maximum at $x_j$. At a local maximum $x_\ep\to x_j$ of $\theta z^\ep-z$ the inequalities above are reversed, and
\[
    c(\ep)\ge \ep z^\ep(x_\ep)-\frac{\ep}{\theta}\Delta z(x_\ep)
      +\left(\frac1{\theta^2}-1\right)V(x_\ep)
    \ge \ep z^\ep(x_\ep)-\frac{\ep}{\theta}\Delta z(x_\ep).
\]
Thus $-\kappa_1\ge z(x_j)-\kappa_j/\theta$, and $\theta\uparrow1$ gives $z(x_j)\le\kappa_j-\kappa_1$.

\end{proof}

\begin{prop}\label{prop:lower-bound}
We have
\[
    \liminf_{\ep\to 0}v^\ep(x)\geq v^*(x)
    \qquad\text{for every }x\in\T^n.
\]
\end{prop}

\begin{proof}
By the definition of $z^\ep$ and \eqref{eq:c-asymptotic-intro}, it is enough to show
\[
 \liminf_{\ep\to 0}z^\ep(x)\geq \min_{1\leq i\leq m}
    \bigl\{\kappa_i-\kappa_1+d(x_i,x)\bigr\}.
\]

Take an arbitrary sequence $\ep_k \downarrow 0$ such that $z^{\ep_k}\to z$ in $C(\T^n)$.
Of course, $z$ solves the cell problem \eqref{eq:cell}, and $z$ can have a local minimum only at one of $x_1,\ldots,x_m$.
Assume that $z$ has local minimum on $Z\subseteq\{x_1,\ldots,x_m\}$, and we write $Z=\{x_{l_1},\ldots,x_{l_p}\}$ for some $1\le p\le m$.
By Lemma \ref{lem:key-lower-bound}, we have
\[
z(x_{l_i})= \kappa_{l_i}-\kappa_1 \qquad \text{for }1\le i \le p.
\]
By the representation formula (see \cite{LMT,Hung-book} for instance),
\[
    z(x)=\min_{1\le i\le m}\{z(x_i)+d(x_i,x)\}.
\]
Fix $x\in\T^n$ and take an index $i\in\{1,\ldots,m\}$ realizing the minimum at $x$, i.e., $z(x)=z(x_i)+d(x_i,x)$. If $x_i\notin Z$, then there exists
$k\ne i$ such that
\[
    z(x_i)=z(x_k)+d(x_k,x_i).
\]
Indeed, otherwise
\[
    \delta:=\min_{k\ne i}
    \{z(x_k)+d(x_k,x_i)-z(x_i)\}>0.
\]
By continuity, for $y$ sufficiently close to $x_i$,
\[
    z(x_k)+d(x_k,y)\ge z(x_i)+\frac{3\delta}{4}
    \qquad\text{for all }k\ne i,
\]
while $d(x_i,y)\le \delta/4$. Hence, 
\[
    z(x_i)+d(x_i,y)
    <
    z(x_k)+d(x_k,y)
    \qquad\text{for all }k\ne i,
\]
and therefore
\[
    z(y)=z(x_i)+d(x_i,y)\ge z(x_i),
\]
which contradicts $x_i\notin Z$.

By the triangle inequality,
\[
    z(x_k)+d(x_k,x)
    \le z(x_k)+d(x_k,x_i)+d(x_i,x)
    =z(x),
\]
while the reverse inequality follows from the representation formula. Thus $x_k$ also realizes the minimum at $x$. If $x_k\notin Z$, we repeat the same argument with $x_k$ in place of $x_i$. Since $d(x_k,x_i)>0$ for $k\ne i$, we have $z(x_k)<z(x_i)$. Repeating the argument, after finitely many steps we reach a point of $Z$. Hence,
\[
\begin{aligned}
    z(x)
    &=\min_{1\le i\le p}
      \{z(x_{l_i})+d(x_{l_i},x)\}\\
    &=\min_{1\le i\le p}
      \{\kappa_{l_i}-\kappa_1+d(x_{l_i},x)\}
    \ge
      \min_{1\le i\le m}
      \{\kappa_i-\kappa_1+d(x_i,x)\}.
\end{aligned}
\]
\end{proof}

\section{The subcritical case}\label{sec:subcritical}

Let $\lambda=\ep^\alpha$ and let $u^\lambda$ solve
\begin{equation}\label{eq:first-order-discount}
    \lambda u^\lambda+|Du^\lambda|^2=V\qquad\text{in }\T^n.
\end{equation}
Since $u^0\ge0$, the comparison principle gives $u^\lambda\le u^0$.
Moreover, $u^0\le CV$ on $\T^n$: near each $x_i$, the line segment $t\mapsto x_i+ty$ gives $d(x_i,x_i+y)\le C|y|^2$, while $V(x_i+y)\ge c|y|^2$; away from the wells, $V$ has a positive lower bound. Thus, for all sufficiently small $\lambda>0$,
\[
    \lambda(1-C\lambda)u^0+|D((1-C\lambda)u^0)|^2-V
    \le \bigl(C\lambda(1-C\lambda)+(1-C\lambda)^2-1\bigr)V=-C\lam V
    \le0.
\]
By the comparison principle for \eqref{eq:first-order-discount}, 
\[
    (1-C\lambda)u^0\le u^\lambda\le u^0,
\]
and therefore
\begin{equation}\label{eq:Clam}
    \|u^\lambda-u^0\|_{L^\infty(\T^n)}\le C\lambda.
\end{equation}
By \cite[Theorem~1.1]{WZ},
\[
    \|w^\ep-u^\lambda\|_{L^\infty(\T^n)}
    \le C\max\left\{\frac{\ep}{\lambda},\ep|\log\ep|\right\}
    \le C\ep^{1-\alpha}.
\]
Combining this with \eqref{eq:Clam} gives \eqref{eq:subcritical-global-rate}.

\begin{lem}\label{lem:w-ep-x-i}
  Fix $\alpha\in (0,1)$.
  We have, for $1\le i \le m$,
  \begin{equation}\label{eq:well-asymptotic-intro}
    w^\ep(x_i)=\kappa_i\ep^{1-\alpha}+o(\ep^{1-\alpha}).
\end{equation} 
\end{lem}

\begin{proof}
Fix $i$ and choose $0<\delta<\lambda_{\min}(A_i)$. 
Put $B_i^-=A_i-\delta I$ and $B_i^+=A_i+\delta I$. 
By \eqref{eq:V-quadratic-expansion}, after reducing $r>0$, for $|y|\le r$,
\[
    y\cdot(B_i^-)^2y\le V(x_i+y)\le y\cdot(B_i^+)^2y.
\]
For a positive definite symmetric matrix $B$ and $\lambda=\ep^\alpha$, set
\begin{equation}\label{def:matrix}
    P_\lambda(B):=\left(B^2+\frac{\lambda^2}{16}I\right)^{1/2}-\frac{\lambda}{4}I.
\end{equation}
Then,
\[
P_\lambda(B)^2+\frac{\lambda}{2}P_\lambda(B)=B^2. 
\]
Define
\[
    Q_{\ep,\lambda,B}(y):=\frac{\ep}{\lambda}\tr P_\lambda(B)+\frac12 y\cdot P_\lambda(B)y.
\]
A direct calculation gives 
\[
\lambda Q_{\ep,\lambda,B}-\ep\Delta Q_{\ep,\lambda,B}+|DQ_{\ep,\lambda,B}|^2=y\cdot B^2y.
\]

Since $u^0(x_i)=0=\min_{\T^n}u^0$, Lemma~\ref{lem:cell problem at local min} gives
\[
    u^0(x_i+y)=\frac12y\cdot A_i y+o(|y|^2), 
\]
where $u^0$ is given by \eqref{func:u0}. 
For fixed $\delta$, reduce $r$ if necessary so that on $|y|=r$,
\[
    \frac12y\cdot B_i^-y<u^0(x_i+y)<\frac12y\cdot B_i^+y.  
\]
By \eqref{eq:subcritical-global-rate}, $w^\ep\to u^0$ uniformly, while $\ep/\lambda=\ep^{1-\alpha}\to0$ and $P_\lambda(B_i^\pm)\to B_i^\pm$. Hence, for all sufficiently small $\ep$,
\[
    Q_{\ep,\lambda,B_i^-}\le w^\ep\le Q_{\ep,\lambda,B_i^+}\qquad\text{on }\partial B(x_i,r).
\]
The comparison principle gives the same inequalities in $B(x_i,r)$. Evaluating at $x_i$,
\[
    \frac{\ep}{\lambda}\tr P_\lambda(B_i^-)\le w^\ep(x_i)\le\frac{\ep}{\lambda}\tr P_\lambda(B_i^+).
\]
Divide by $\ep^{1-\alpha}=\ep/\lambda$, let $\ep\downarrow0$, and then $\delta\downarrow0$. This proves \eqref{eq:well-asymptotic-intro}.
\end{proof}

\begin{proof}[Proof of Theorem {\rm\ref{thm: sum up} (i)}]
Since we already have \eqref{eq:subcritical-global-rate} holds, 
we only need to prove an estimate from below. 
Fix $1/2\le \alpha<1$.
Then, \eqref{eq:subcritical-global-rate} implies
\[
\|w^\ep-u^0\|_{L^\infty(\T^n)}\le C\ep^{1-\alpha}.
\]
In light of \eqref{eq:well-asymptotic-intro}, for $1\le i\le m$ and $\ep>0$ sufficiently small,
\[
|w^\ep(x_i)-u^0(x_i)| = |w^\ep(x_i)|=\kappa_i\ep^{1-\alpha}+o(\ep^{1-\alpha}) \geq \frac{\kappa_i}{2}\ep^{1-\alpha}.
\]
Letting $c=\max_{1\le i\le m}\kappa_i/2$, we conclude that
\[
    c\ep^{1-\alpha}\le \|w^\ep-u^0\|_{L^\infty(\T^n)}\le C\ep^{1-\alpha}.
\qedhere
\]
\end{proof}

\section{The supercritical case}\label{sec:supercritical}

\subsection{Proof of Theorem \ref{thm: sum up}(ii)}
\begin{proof}[Proof of Theorem {\rm\ref{thm: sum up} (ii)}]
Set $z^\ep:=w^\ep+c(\ep)/\ep^\alpha$. Then
\begin{equation}\label{eq:z-general-alpha-pde}
    \ep^\alpha z^\ep-\ep\Delta z^\ep+|Dz^\ep|^2=V+c(\ep)\qquad\text{in }\T^n.
\end{equation}
Normalize $\chi^\ep$ in \eqref{eq:ergodic} by $\min_{\T^n}\chi^\ep=0$. 
By \cite[Lemma~2]{AIPSM}, there exists $C>0$ independent of $\ep\in (0,1)$ such that
\[
\|\chi^\ep\|_{L^\infty}+\|D\chi^\ep\|_{L^\infty}\le C.
\]
We have $\chi^\ep$ is a supersolution, and $\chi^\ep-\|\chi^\ep\|_{L^\infty}$ is a subsolution of \eqref{eq:z-general-alpha-pde}, respectively.
By the comparison principle and the Lipschitz estimate for $z^\ep$ (e.g., \cite[Lemma~2.2]{WZ}),
\[
    \|z^\ep\|_{L^\infty(\T^n)}+\|Dz^\ep\|_{L^\infty(\T^n)}\le C.
\]

Under \eqref{eq:unique-j}, we use \cite[Corollary~1]{AIPSM} to deduce that
\[
    \chi^\ep-\chi^\ep(x_j)\longrightarrow d(x_j,\cdot)\qquad\text{uniformly on }\T^n
\ \text{as} \ \ep\to0. 
\]
The one-dimensional case was studied first in \cite{JKM}.
Since $\min\chi^\ep=0$ and $\min d(x_j,\cdot)=d(x_j,x_j)=0$, it follows that $\chi^\ep(x_j)\to0$, and hence
\begin{equation}\label{eq:chi-limit}
    \chi^\ep\longrightarrow d(x_j,\cdot)\qquad\text{uniformly on }\T^n 
    \ \text{as} \ \ep\to0. 
\end{equation}
In particular, $z^\ep\le\chi^\ep$ yields 
\[
\limsup_{\ep\to0}z^\ep(x)\le d(x_j,x).
\]

For the reverse inequality, we use the nonlinear adjoint method introduced in \cite{EvansAdjoint}; see also \cite{CGMT, MitakeTran, TranAdjoint}. 
More precisely, we linearize the Hamilton--Jacobi equation \eqref{eq:z-general-alpha-pde} at $z^\ep$ and consider the corresponding adjoint equation in the distributional sense. 
Fix $x\in\T^n$ and let $\theta_x^\ep$ be the nonnegative distributional solution of
\begin{equation*}
\begin{cases}
      \ep^\alpha\theta-\Div(2Dz^\ep\theta)=\ep\Delta\theta+\ep^\alpha\delta_x \qquad \text{in } \T^n,\\  
      \int_{\T^n}\theta\,dy=1, 
\end{cases}
\end{equation*}
where we denote by $\delta_x$ the Dirac measure at $x$. 
Define the probability measure $\eta_x^\ep$ on $\T^n\times\R^n$ by 
\[d\eta_x^\ep(y,q):=\theta_x^\ep(y)\,dy\,\delta_{2Dz^\ep(y)}(dq).
\]
For every fixed $\phi\in C^2(\T^n)$,
\begin{equation}\label{eq:adjoint-identity-super}
    \int_{\T^n\times\R^n}
    (q\cdot D\phi-\ep\Delta\phi)\,d\eta_x^\ep
    =
    \ep^\alpha
    \left(
    \phi(x)-\int_{\T^n}\phi\,d\theta_x^\ep
    \right)
    =o(\ep).
\end{equation}
Taking $\phi=z^\ep$ in the first equality of
\eqref{eq:adjoint-identity-super}, and using
$q=2Dz^\ep$ and \eqref{eq:z-general-alpha-pde}, we obtain
\[
\begin{aligned}
\ep^\alpha\left(
z^\ep(x)-\int_{\T^n}z^\ep \theta_x^\ep\,dy
\right)
&=\int_{\T^n\times\R^n}(q\cdot Dz^\ep-\ep\Delta z^\ep)\, d\eta_x^\ep
\\ &=
\int_{\T^n}
\left(2|Dz^\ep|^2-\ep\Delta z^\ep\right)
\theta_x^\ep\,dy\\
&=\int_{\T^n}(|Dz^\ep|^2+V + c(\ep)-\ep^\alpha z^\ep)\theta^\ep_x\,dy
\\
&=
\int_{\T^n\times\R^n}
\left(L(y,q)+c(\ep)-\ep^\alpha z^\ep(y)\right)
\,d\eta_x^\ep .
\end{aligned}
\]
Therefore,
\begin{equation}\label{eq:action-super}
    \int_{\T^n\times\R^n}
    \left(L(y,q)+c(\ep)\right)\,d\eta_x^\ep
    =
    \ep^\alpha z^\ep(x)
    =
    o(\ep),
\end{equation}
where the last equality follows from the uniform bound on $z^\ep$ and
$\alpha>1$.
 Since $c(\ep)=-\ep\kappa_j+o(\ep)$, we have 
 \begin{align*}
&\int_{\T^n\times\R^n}V(y)\,d\eta^\ep_x\le 
 \int_{\T^n\times\R^n} L(y,q)\,d\eta_x^\ep=\ep\kappa_j+o(\ep)\\
 \Longrightarrow & \ \int_{\T^n\times\R^n} V\,d\eta_x^\ep=\int_{\T^n} V\,d\theta^\ep_x=O(\ep).
 \end{align*}
 Hence, every weak limit of $\pi_\#\eta_x^\ep=\theta^\ep_x\,dy$ is supported on $\{x_1,\ldots,x_m\}$.

Fix $\delta>0$ so small that $A_i-\delta I>0$ for every $i$. 
Choose $r>0$ small enough such that $\{B(x_i,r)\}_{i=1}^m$ are disjoint and $|(A_i-\delta I)y|^2\le V(x_i+y)$ in $B(x_i,r)$, and set 
\[
s_i(y)=\frac12y\cdot(A_i-\delta I)y.
\]
Choose $s^*>0$ sufficiently small such that $\{y:s_i(y)<s^*\}\Subset B(x_i,r)$, and $h\in C_c^\infty([0,\infty))$ such that $h(s)=s$ near $0$, $|h'|\le1$, and $\operatorname{supp}h\subset[0,s^*)$. 
Define $\psi_\delta(x_i+y)=h(s_i(y))$ in $B(x_i,r)$ for $1\le i \le m$, and $\psi_\delta=0$ elsewhere. 
Then, $\psi_\delta\in C^2(\T^n)$, $D^2\psi_\delta(x_i)=A_i-\delta I$ for $1\le i \le m$, and
\[
    |D\psi_\delta|^2\le V \qquad \text{in } \T^n.
\]
By the Cauchy-Schwarz inequality,
\[
L(y,q)=|q|^2/4+V(y)\ge q\cdot D\psi_\delta(y).
\]
If $\pi_\#\eta_x^{\ep_k}=\theta^{\ep_k}_x\rightharpoonup\sum_{i=1}^m \mu_i\delta_{x_i}$, where $\mu_i\geq 0$ and $\sum_{i=1}^m \mu_i=1$, then \eqref{eq:adjoint-identity-super}, and \eqref{eq:action-super} give
\[
    \kappa_j\ge\sum_{i=1}^m \mu_i\Delta\psi_\delta(x_i)=\sum_{i=1}^m \mu_i(\kappa_i-n\delta).
\]
Letting $\delta\downarrow0$ and using \eqref{eq:unique-j}, we obtain $\mu_j=1$. 
Thus,
\begin{equation}\label{eq:eta-concentration}
    \pi_\#\eta_x^\ep\rightharpoonup\delta_{x_j}.
\end{equation}

It remains to compare $z^\ep$ and $\chi^\ep$. Put $f^\ep=z^\ep-\chi^\ep$. From convexity,
\[
    |Dz^\ep|^2-|D\chi^\ep|^2\le 2Dz^\ep\cdot Df^\ep,
\]
and subtraction of the equations for $z^\ep$ and $\chi^\ep$ gives
\[
    \ep^\alpha z^\ep-\ep\Delta f^\ep+2Dz^\ep\cdot Df^\ep\ge0.
\]
Multiplying by $\theta_x^\ep$ and using \eqref{eq:adjoint-identity-super} with $\phi=f^\ep$ yields
\[
    z^\ep(x)\ge \chi^\ep(x)-\int_{\T^n\times\R^n}\chi^\ep\,d\eta_x^\ep.
\]
By \eqref{eq:chi-limit} and \eqref{eq:eta-concentration}, the last integral converges to $d(x_j,x_j)=0$. 
Hence, 
\[
\liminf_{\ep\to0}z^\ep(x)\ge d(x_j,x),
\]
which together with the upper bound proves the assertion for $\alpha>1$.
\end{proof}

We now give a sketch of the construction of the nonconvergence result.

\subsection{A nonconvergence result in one dimension}
Under the assumption $V\in C^2$, the construction of Liu--Tran--Yu \cite[Theorem~1.2]{LTY} can be adapted to give a nonconvergence example for every fixed $\alpha>1$. 

We explain the modifications of the construction and the localization argument in \cite{LTY} and give a sketch of the proof of Theorem \ref{thm:nonconvergence}.
\begin{proof}[Sketch of the proof of Theorem {\rm\ref{thm:nonconvergence}}]
Recall that $w^\ep$ solves
\[
\ep^\alpha w^\ep-\ep(w^\ep)''+|(w^\ep)'|^2=V
 \qquad\text{in }\T.
\]
Then, $z^\ep$ solves
\[
\ep^\alpha z^\ep-\ep(z^\ep)''+|(z^\ep)'|^2=V+c(\ep)
 \qquad\text{in }\T.
\]
Choose
\[
 1<\beta<\min\{\alpha,3/2\},\qquad a=\frac12,
 \qquad A>0.
\]
For $0<r<\rho$, where $\rho\in(0,1/20)$ will be chosen sufficiently
small, set
\[
 \omega(r):=\sin\bigl(\log\log(1/r)\bigr).
\]
Denote by
\[
g(y):=
\begin{cases}
|y|^{2\beta}\omega(|y|)\qquad &y\ne0,\\
0 \qquad &y=0.
\end{cases}
\]
Since $2\beta>2$, we have $g\in C^2(( -\rho,\rho))$ and
$g(0)=g'(0)=g''(0)=0$.
Replace the oscillatory quartic perturbations in
\cite[Section~1.2, equations~(3)--(5)]{LTY} by $\pm Ag$, and choose, for $ |y|<\rho$,
\begin{equation*}
 \begin{aligned}
 V(y)&=y^2+A|y|^{2\beta}\omega(|y|),\\
 V(a+y)&=y^2-A|y|^{2\beta}\omega(|y|).
 \end{aligned}
\end{equation*}
Choose $\rho$ so that $A\rho^{2\beta-2}\leq1/2$, and extend $V$ smoothly and positively away from these two neighborhoods. Then
\[
\begin{cases}
  V\in C^2(\T;[0,\infty)),\\
  \{V=0\}=\{0,a\},\\ 
  V''(0)=V''(a)=2.
\end{cases}
\]
In particular, both wells are nondegenerate and $\kappa_0=\kappa_a=1$.
Changing the power of the perturbation is also the idea used in
\cite[Remark~4.1]{LTY}; here we use a power between $2$ and $3$ to
retain $C^2$ regularity and obtain an energy gap larger than the discount.

Let $E_0^D(\ep)$ and $E_a^D(\ep)$ be the first Dirichlet eigenvalues of
\[
 P_\ep:=-\ep^2\frac{d^2}{dx^2}+V
\]
on $(-\rho,\rho)$ and $(a-\rho,a+\rho)$, respectively. The harmonic
approximation in \cite[Lemma~2.1]{LTY}, with $|y|^{2\beta}$ replacing
$y^4$, gives
\begin{equation}\label{eq:rem-C2-local-energies}
 \begin{aligned}
 E_0^D(\ep)&=\ep+A M_\beta\omega(\sqrt\ep)\ep^\beta
                    +o(\ep^\beta),\\
 E_a^D(\ep)&=\ep-A M_\beta\omega(\sqrt\ep)\ep^\beta
                    +o(\ep^\beta),
 \end{aligned}
\end{equation}
where $M_\beta:=\frac{\Gamma(\beta+1/2)}{\sqrt\pi}>0$.
To verify \eqref{eq:rem-C2-local-energies}, set $x=x_i+\sqrt\ep\,y$. The local
operator becomes
\[
 \ep\left(-\frac{d^2}{dy^2}+y^2
       \mathbin{\pm}A\ep^{\beta-1}|y|^{2\beta}
                           \omega(\sqrt\ep|y|)\right).
\]
Its normalized ground state converges in $L^2(\R)$, after extension by
zero, to $\psi_0(y)=\pi^{-1/4}e^{-y^2/2}$. 
The quadratic lower bound and the weighted estimate used in \cite[proof of Lemma~2.1, equations~(13)--(17)]{LTY} give uniformly Gaussian tails, and therefore uniform control of the $2\beta$-moments. 
Moreover, for fixed $y \ne 0$,
\[
 \omega(\sqrt\ep|y|)-\omega(\sqrt\ep)\longrightarrow0.
\]
The Rayleigh upper and lower bounds from that proof give
the first correction
\[
 \mathbin{\pm}A\ep^\beta\omega(\sqrt\ep)
       \int_{\R}|y|^{2\beta}\psi_0(y)^2\,dy
 +o(\ep^\beta).
\]
The integral equals $M_\beta$, proving
\eqref{eq:rem-C2-local-energies}.

We next account for the discount term. 
We have $E(\ep)=-c(\ep)$ is the principal eigenvalue of $P_\ep$, and $0<E(\ep)\leq C\ep$.
The Hopf--Cole transform $U^\ep=e^{-z^\ep/\ep}$ satisfies
\begin{equation}\label{eq:rem-C2-perturbed-operator}
 \left(-\ep^2\frac{d^2}{dx^2}+W_\ep\right)U^\ep
       =E(\ep)U^\ep,
 \qquad W_\ep:=V-\lambda z^\ep.
\end{equation}
Because $U^\ep>0$, it is a principal eigenfunction of the operator in
\eqref{eq:rem-C2-perturbed-operator}. 
As $\lam=\ep^\al$,
\[
 \|W_\ep-V\|_{L^\infty}\leq C\ep^\alpha=o(\ep^\beta).
\]
The Rayleigh characterization of the local Dirichlet eigenvalues
$\widetilde E_i^D(\ep)$ for this perturbed operator gives
\begin{equation}\label{eq:rem-C2-perturbed-energies}
 |\widetilde E_i^D(\ep)-E_i^D(\ep)|\leq C\ep^\alpha
                  =o(\ep^\beta),\qquad i\in\{0,a\}.
\end{equation}
In particular, the discount perturbation preserves the alternating
energy gap in \eqref{eq:rem-C2-local-energies}.

Choose the same sequences as in \cite[Section~4, proof of
Theorem~1.2]{LTY}:
\begin{equation}\label{eq:rem-C2-sequences}
 \ep_k^+:=\exp\!\left[-2\exp\!\left(\frac\pi2+2\pi k\right)\right],
 \qquad
 \ep_k^-:=\exp\!\left[-2\exp\!\left(\frac{3\pi}2+2\pi k\right)\right].
\end{equation}
Then, $\omega(\sqrt{\ep_k^+})=1$ and
$\omega(\sqrt{\ep_k^-})=-1$. By
\eqref{eq:rem-C2-local-energies} and
\eqref{eq:rem-C2-perturbed-energies}, there exists $c_0>0$ such that,
for all sufficiently large $k$,
\begin{equation*}
 \begin{aligned}
 \widetilde E_0^D(\ep_k^+)-\widetilde E_a^D(\ep_k^+)
        &\geq c_0(\ep_k^+)^\beta,\\
 \widetilde E_a^D(\ep_k^-)-\widetilde E_0^D(\ep_k^-)
        &\geq c_0(\ep_k^-)^\beta.
 \end{aligned}
\end{equation*}

We explain why the proof of \cite[Lemma~3.1]{LTY} applies to the
$\ep$-dependent potential $W_\ep$. Normalize the eigenfunction in
\eqref{eq:rem-C2-perturbed-operator} by
$\Psi^\ep=U^\ep/\|U^\ep\|_{L^2(\T)}$.
First, outside fixed neighborhoods of the two wells, $W_\ep$ has a
uniform positive lower bound for all small $\ep$; globally,
$W_\ep\geq-C\ep^\alpha$. Since $E(\ep)=O(\ep)$ and $\alpha>1$,
the weighted eigenfunction identity in
\cite[proof of Lemma~3.1, Step~1]{LTY} gives
\begin{equation}\label{eq:rem-C2-barrier}
 \int_{\T\setminus O_r}|\Psi^\ep|^2\,dx\leq Ce^{-s_r/\ep},
 \qquad O_r:=(-r,r)\cup(a-r,a+r),
\end{equation}
for each fixed sufficiently small $r>0$ and some $s_r>0$. 

Second, apply the partition-of-unity argument in
\cite[proof of Lemma~3.1, Step~2, equations~(26)--(30)]{LTY}.
Write $x_{\mathrm{low}}$ and $x_{\mathrm{high}}$ for the lower- and
higher-energy wells along either sequence in
\eqref{eq:rem-C2-sequences}. If $m_{\mathrm{low}}^\ep$ and
$m_{\mathrm{high}}^\ep$ are the corresponding cutoff $L^2$-masses,
the IMS identity and \eqref{eq:rem-C2-barrier} give
\[
 E(\ep)\geq\widetilde E_{\mathrm{low}}^D m_{\mathrm{low}}^\ep
             +\widetilde E_{\mathrm{high}}^D m_{\mathrm{high}}^\ep
             -Ce^{-s/\ep},
 \qquad
 m_{\mathrm{low}}^\ep+m_{\mathrm{high}}^\ep=1+O(e^{-s/\ep}).
\]
Since $E(\ep)\leq\widetilde E_{\mathrm{low}}^D$ and the gap is at
least $c_0\ep^\beta$, it follows that
\[
 c_0\ep^\beta m_{\mathrm{high}}^\ep\leq Ce^{-s/\ep}.
\]
Absorbing the polynomial factor into the exponential proves that the
mass in the higher-energy well is exponentially small. Thus the
$\ep^2$ gap in the statement of \cite[Lemma~3.1]{LTY} is replaced here
by the gap $c_0\ep^\beta$; the proof uses it in precisely this inequality.

Finally, the passage from these mass estimates to the pointwise ratio
in \cite[proof of Lemma~3.1, Step~3, equations~(32)--(35)]{LTY} is
uniform for $W_\ep$. On each fixed rescaled interval $|y|\leq R$,
\[
 \left|\frac{W_\ep(x_i+\sqrt\ep\,y)}\ep\right|
 \leq C_R+C\ep^{\alpha-1},
 \qquad 0 \le \frac{E(\ep)}\ep\leq C.
\]
Hence the rescaled local estimates and Harnack inequality have uniform
constants. Also, the eigenfunction energy identity gives
\[
 \int_{\T}V|\Psi^\ep|^2\,dx
 \leq E(\ep)+\lambda\|z^\ep\|_{L^\infty}\leq C\ep.
\]
Together with the mass concentration in the lower-energy well and its
quadratic lower bound, this places a fixed positive amount of mass
within distance $R\sqrt\ep$ of that well for some fixed large $R$.
As in the cited Step~3, the lower-well value is therefore at least
$c\ep^{-1/4}$, whereas the higher-well value is exponentially small
up to a polynomial factor. Consequently, there exist $C,\sigma>0$
independent of $k$ such that
\begin{equation*}\label{eq:rem-C2-ratios}
 \frac{U^{\ep_k^+}(0)}{U^{\ep_k^+}(a)}
      \leq Ce^{-\sigma/\ep_k^+},
 \qquad
 \frac{U^{\ep_k^-}(a)}{U^{\ep_k^-}(0)}
      \leq Ce^{-\sigma/\ep_k^-}.
\end{equation*}
Recalling that $z^\ep=-\ep\log U^\ep$, we conclude 
\[
 \limsup_{k\to\infty}
       \bigl(z^{\ep_k^+}(a)-z^{\ep_k^+}(0)\bigr)\leq-\sigma<0,
 \qquad
 \liminf_{k\to\infty}
       \bigl(z^{\ep_k^-}(a)-z^{\ep_k^-}(0)\bigr)\geq\sigma>0.
\]
Thus, $\{z^\ep\}$ does not converge in $C(\T)$. 
\end{proof}

\section{Concentration of the Gibbs measures}\label{sec:cole-hopf}

\begin{proof}[Proof of Theorem {\rm\ref{thm:measure}}]

We first note that adding constants to $w^\ep$ does not change $\rho^\ep$ defined in \eqref{eq:rho}. 

When $\alpha=1$, thanks to Theorem~\ref{thm:critical-case}, $\widetilde w^\ep:=w^\ep$ converges to $v^*$ uniformly in $\T^n$ as $\ep \to 0$, and $v^*$ has a unique global minimum at $x_j$. 

When $\alpha>1$, in light of Theorem~\ref{thm: sum up}, $\widetilde w^\ep:=w^\ep+\frac{c(\ep)}{\ep^\al}$ converges to $d(x_j,\cdot)$ uniformly in $\T^n$ as $\ep \to 0$, and again, $d(x_j,\cdot)$ has a unique global minimum at $x_j$.

In either case, let $f$ denote the uniform limit of $\{\widetilde w^\ep\}$.
Fix a neighborhood $U$ of $x_j$. Since $x_j$ is the unique global
minimum of $f$, there exists $\delta>0$ such that
\[
    f(x)\ge f(x_j)+4\delta
    \qquad\text{for }x\in U^c.
\]
By continuity of $f$, we may choose $r>0$ such that
$B(x_j,r)\subset U$ and
\[
    f(x)\le f(x_j)+\frac{\delta}{2}
    \qquad\text{for }x\in B(x_j,r).
\]
By uniform convergence, for all sufficiently small $\ep>0$,
\[
\begin{cases}
\widetilde w^\ep(x)\ge
\widetilde w^\ep(x_j)+2\delta,
& x\in U^c,\\[1mm]
\widetilde w^\ep(x)\le
\widetilde w^\ep(x_j)+\delta,
& x\in B(x_j,r).
\end{cases}
\]
Therefore, as $\ep\to 0$,
\[
    \frac{\int_{U^c}e^{-\widetilde w^\ep/\ep}\,dx}{\int_{\T^n}e^{-\widetilde w^\ep/\ep}\,dx}\le C e^{-\delta/\ep}\longrightarrow0,
\]
which proves \eqref{eq:rho-concentration-intro} for $\alpha\ge1$.

We now consider the subcritical case $0<\alpha<1$.
Recall that $\lambda=\ep^\alpha$. 
Fix $\delta>0$ sufficiently small that $\tr(A_j+\delta I)<\tr(A_i-\delta I)$ for every $i\ne j$. 
The comparison in Section~\ref{sec:subcritical} gives, for $|y|<r$,
\[
    w^\ep(x_j+y)\le \frac{\ep}{\lambda}\tr P_\lambda(A_j+\delta I)+\frac12y\cdot P_\lambda(A_j+\delta I)y,
\]
where $P_\lambda(B)$ is the matrix given by \eqref{def:matrix}
and, for $i\ne j$,
\[
    w^\ep(x_i+y)\ge \frac{\ep}{\lambda}\tr P_\lambda(A_i-\delta I)+\frac12y\cdot P_\lambda(A_i-\delta I)y.
\]
Consequently,
\[
    \int_{\T^n}e^{-w^\ep/\ep}\,dx\ge c\ep^{n/2}\exp\left(-\frac{\tr P_\lambda(A_j+\delta I)}{\lambda}\right),
\]
while for $i\ne j$,
\[
    \int_{B(x_i,r)}e^{-w^\ep/\ep}\,dx\le C\ep^{n/2}\exp\left(-\frac{\tr P_\lambda(A_i-\delta I)}{\lambda}\right).
\]
Since $P_\lambda(B)\to B$ as $\lambda\to0$, the ratio of the latter quantity to the former tends to $0$. 
On $E=\T^n \setminus \bigcup_{i=1}^m B(x_i,r)$, we have $u^0>0$.
As $w^\ep\to u^0$ uniformly thanks to \eqref{eq:subcritical-global-rate},  $w^\ep\ge a>0$ on $E$ for small $\ep$. Hence, 
\[
    \frac{\int_E e^{-w^\ep/\ep}\,dx}{\int_{\T^n}e^{-w^\ep/\ep}\,dx}
    \le C\ep^{-n/2}\exp\left(-\frac a\ep\right)\longrightarrow0.
\]
Thus, \eqref{eq:rho-concentration-intro} holds also for $0<\alpha<1$.
\end{proof}

\begin{rem}
The results in this paper suggest several natural directions for further study.
\begin{enumerate}[(i)]
\item It would be interesting to understand the corresponding selection problem when the wells of $V$ are degenerate. In this case, $c(\ep)$ is no longer expected to be of order $\ep$, and the critical balance between discount and viscosity should depend on the order of degeneracy.

\item Suppose that the minimum of $\kappa_i$ is attained at more than one
well. For a fixed $\alpha>1$, can one characterize the subsequential limits of $z^\ep$ as $\ep\to0$? In particular, under what additional assumptions on $V$
does the normalized family converge?

\item One may ask for an analog of Theorem \ref{thm:critical-case} for a Tonelli Hamiltonian whose Aubry set consists of finitely many hyperbolic equilibria or periodic orbits. For the pure vanishing viscosity problem, the local second-order information is related to the stable bundle of the linearized Hamiltonian flow; see \cite{AIPSM}. It would be interesting to determine how these local quantities interact with the vanishing discount term, and whether they lead to an analog of $v^*(x)$.

\item Finally, for a general coercive and convex Hamiltonian $H\in C^2(\T^n\times \R^n)$, is the family $\{v^\ep\}_{\ep\in (0,1)}$ convergent as $\ep \to 0$? Here, $v^\ep$ solves
\[
\ep v^\ep + H(x, Dv^\ep) - \ep \Delta v^\ep =0 \qquad \text{in } \T^n.
\]
\end{enumerate}
\end{rem}

\appendix
\section{Some known results}\label{app:known-results}
\begin{lem}\label{lem:cell problem at local min}
Let $v\in C(\T^n)$ be a solution to the cell problem \eqref{eq:cell}.  If $v$ has a local min at $x_i$ for some $i\in\{1,\ldots,m\}$, then $v$ is $C^2$ in a neighborhood of $x_i$ and
\[
    D^2v(x_i)=A_i.
\]
\end{lem}

\begin{proof}
Set $\phi=-v$. By the local stable-manifold argument in
\cite[Lemma~1]{AIPSM}, $\phi$ is $C^2$ near $x_i$. The same proof
applies for $V\in C^2$, since the associated Hamiltonian vector field
is then $C^1$ and its local stable manifold, whose graph represents
$D\phi$, is $C^1$. Since $x_i$ is a local minimum of $v$, $D^2v(x_i)\ge0$. A second-order
expansion of $|Dv|^2=V$ at $x_i$ gives
\[
    \bigl(D^2v(x_i)\bigr)^2
    =\frac12D^2V(x_i)=A_i^2.
\]
Hence $D^2v(x_i)=A_i$.
\end{proof}

\begin{thm}\label{thm:asymptotic-c-ep}
We have
\[
    \lim_{\ep\to0}\frac{c(\ep)}{\ep}=-\min_{1\le i\le m}\kappa_i.
\]
\end{thm}

We refer the reader to \cite[Appendix A]{LTY} for a simple PDE proof of Theorem \ref{thm:asymptotic-c-ep}, whose argument extends directly to the present multidimensional setting.
Lemma~\ref{lem:cell problem at local min} provides the multidimensional
analog of the local second-order information used in
\cite[Lemma~A.2]{LTY}, and the argument of \cite[Lemma~A.4]{LTY}
then yields the result.
For regularity of the viscous effective Hamiltonian and an $O(\ep)$ estimate for its convergence in the vanishing-viscosity limit, see \cite{TuZhangEH}.

\section{Another proof of Proposition~\ref{prop:upper-bound} by stochastic control}\label{appendix: another proof}
\begin{proof}[Another proof of Proposition~ {\rm\ref{prop:upper-bound}}]
We give a second proof based on the stochastic control representation.  For
\eqref{eq:discount-viscous},
\[
    v^\ep(x)=\inf_{\beta}\mathbb E_x\left[
    \int_0^\infty e^{-\ep t}
    L(X_t,\beta_t)\,dt\right].
\]
Here,
\[
\begin{cases}
    dX_t=\beta_t\,dt+\sqrt{2\ep}\,dW_t,\\
    X_0=x,
\end{cases}
\]
where $W$ is a standard $n$-dimensional Brownian motion; see, for instance,
\cite{FS}.  In particular, since $L\ge0$, and the choice $\beta\equiv0$ gives
an upper bound, we have
\begin{equation}\label{eq:stochastic-crude-bound}
    0\le v^\ep\le \frac{\|V\|_{L^\infty(\T^n)}}{\ep}.
\end{equation}

We first estimate the value function near a fixed well $x_i$.  Fix
$\delta>0$ and put $B:=A_i+\delta I$.  
After reducing $r>0$, if necessary, for $|y|\le r$,
\begin{equation}\label{eq:stochastic-local-V-upper}
    V(x_i+y)\le y\cdot B^2y.
\end{equation}
For $|y|\le r/2$, use the feedback control
$\beta_t=-2B(X_t-x_i)$ until the first exit time
\[
    \tau_r:=\inf\{t\ge0:|X_t-x_i|=r\}.
\]
Writing $Y_t=X_t-x_i$ in the coordinate ball, we have, up to $\tau_r$,
\begin{equation}\label{eq:OU-local-control}
\begin{cases}
    dY_t=-2BY_t\,dt+\sqrt{2\ep}\,dW_t,\\
    Y_0=y. 
\end{cases}
\end{equation}
By \eqref{eq:stochastic-local-V-upper},
\[
    L(x_i+Y_t,-2BY_t)
    =Y_t\cdot B^2Y_t+V(x_i+Y_t)
    \le 2Y_t\cdot B^2Y_t.
\]
Set $\Phi(y)=\frac12y\cdot By$.  The generator of
\eqref{eq:OU-local-control} satisfies
\[
    \mathcal L_\ep\Phi(y)
    =-2y\cdot B^2y+\ep\tr B.
\]
Applying It\^o's formula to $e^{-\ep(t\wedge\tau_r)}\Phi(Y_{t\wedge\tau_r})$
and then letting $t\to\infty$ gives
\begin{align}
&\mathbb E_y\int_0^{\tau_r}e^{-\ep t}
       2Y_t\cdot B^2Y_t\,dt \notag\\
&\quad =\Phi(y)-\mathbb E_y\left[e^{-\ep\tau_r}\Phi(Y_{\tau_r})\right]
 +\tr B\,\mathbb E_y\left[1-e^{-\ep\tau_r}\right]
 -\ep\mathbb E_y\int_0^{\tau_r}e^{-\ep t}\Phi(Y_t)\,dt \notag\\
&\quad \le \Phi(y)+\tr B .
\label{eq:OU-running-cost}
\end{align}

We next show that the discounted exit term is negligible.  The solution of
\eqref{eq:OU-local-control} before $\tau_r$ is
\[
    Y_t=e^{-2Bt}y+Z_t,
    \qquad
    Z_t:=\sqrt{2\ep}\int_0^t e^{-2B(t-s)}\,dW_s.
\]
Since $B$ is positive definite, $\|e^{-2Bt}\|\le1$.  Thus, for
$|y|\le r/2$,
\[
    \{\tau_r\le T\}
    \subset
    \left\{\sup_{0\le t\le T}|Z_t|\ge r/2\right\}.
\]
We estimate the latter event on unit time intervals.  For an integer $k\ge0$
and $0\le s\le1$,
\[
    Z_{k+s}=e^{-2Bs}Z_k
    +\sqrt{2\ep}\,e^{-2Bs}M_s^{(k)},
    \qquad
    M_s^{(k)}:=\int_0^s e^{2Bu}\,dW_u^{(k)},
\]
where $W_u^{(k)}:=W_{k+u}-W_k$ is again a standard Brownian motion.  Let
$b_*:=\lambda_{\min}(B)$.  The covariance matrix of $Z_k$ is bounded above by
$\ep(2b_*)^{-1}I$.  After diagonalizing $B$ by an orthogonal matrix, the
coordinates of $M^{(k)}$ are continuous Gaussian martingales whose quadratic
variations on $[0,1]$ are bounded by
\[
    Q:=\max_{\mu\in\sigma(B)}\int_0^1e^{4\mu s}\,ds<\infty.
\]
The Gaussian tail bound and the reflection principle, followed by a union
bound over the $n$ coordinates, therefore give constants $C,c>0$, independent
of $k$ and $\ep$, such that
\begin{equation}\label{eq:OU-unit-exit}
    \mathbb P\left(\sup_{0\le s\le1}|Z_{k+s}|\ge r/2\right)
    \le Ce^{-c/\ep}.
\end{equation}
Indeed, if the event on the left occurs, then either $|Z_k|\ge r/4$ or
$\sup_{0\le s\le1}|M_s^{(k)}|\ge r/(4\sqrt{2\ep})$, and each of these two
events has the asserted bound.

For $0<\ep<1/2$, let $R_\ep=2|\log\ep|$.  Covering
$[0,R_\ep/\ep]$ by unit intervals and using \eqref{eq:OU-unit-exit}, we obtain
\[
    \sup_{|y|\le r/2}
    \mathbb P_y\left(\tau_r\le\frac{R_\ep}{\ep}\right)
    \le C\left(1+\frac{R_\ep}{\ep}\right)e^{-c/\ep}.
\]
Consequently,
\begin{align*}
    \sup_{|y|\le r/2}\mathbb E_y[e^{-\ep\tau_r}]
    &\le
    C\left(1+\frac{R_\ep}{\ep}\right)e^{-c/\ep}
    +e^{-R_\ep}
    =o(\ep).
\end{align*}
The dynamic programming principle, \eqref{eq:stochastic-crude-bound}, and
\eqref{eq:OU-running-cost} now yield, uniformly for $|y|\le r/2$,
\begin{equation}\label{eq:stochastic-local-value}
    v^\ep(x_i+y)
    \le \tr B+\frac12y\cdot By+o(1).
\end{equation}

We now fix $x\in\T^n$.  Given $\eta>0$, choose an absolutely continuous curve
$\gamma:[0,T]\to\T^n$ with $\gamma(0)=x$, $\gamma(T)=x_i$, and
\[
    \int_0^T L(\gamma(t),\dot\gamma(t))\,dt
    \le d(x,x_i)+\eta.
\]
On $[0,T]$ use the control $\beta_t=\dot\gamma(t)$.  Then
$X_t=\gamma(t)+\sqrt{2\ep}W_t$ modulo $\T^n$.  Since $T$ is fixed, dominated
convergence gives
\begin{equation}\label{eq:stochastic-path-cost}
    \mathbb E_x\int_0^T e^{-\ep t}
    L(X_t,\dot\gamma(t))\,dt
    \le d(x,x_i)+\eta+o(1).
\end{equation}
At time $T$, in a lift centered at $x_i$, write
$Y_T=\sqrt{2\ep}W_T$.  On $\{|Y_T|\le r/2\}$ we use
\eqref{eq:stochastic-local-value}, whereas on the complementary event we use
\eqref{eq:stochastic-crude-bound}.  Since
\[
    \mathbb E\left[\frac12Y_T\cdot BY_T\right]
    =\ep T\tr B
\]
and
\[
    \mathbb P(|Y_T|>r/2)\le Ce^{-c/\ep},
\]
we obtain
\begin{equation}\label{eq:stochastic-terminal-value}
    \mathbb E_x[v^\ep(X_T)]\le \tr B+o(1).
\end{equation}
Combining \eqref{eq:stochastic-path-cost} and
\eqref{eq:stochastic-terminal-value} in the dynamic programming principle,
\[
    v^\ep(x)
    \le \mathbb E_x\left[
    \int_0^T e^{-\ep t}L(X_t,\dot\gamma(t))\,dt
    +e^{-\ep T}v^\ep(X_T)\right],
\]
gives
\[
    \limsup_{\ep\to0}v^\ep(x)
    \le d(x,x_i)+\eta+\tr B.
\]
Finally let $\eta\downarrow0$ and $\delta\downarrow0$.  Since
$\tr B\to\tr A_i=\kappa_i$ and, by the reversibility of $L$,
$d(x,x_i)=d(x_i,x)$, we conclude that
\[
    \limsup_{\ep\to0}v^\ep(x)
    \le \kappa_i+d(x_i,x).
\]
Taking the minimum over $i$ proves Proposition~\ref{prop:upper-bound}.
\end{proof}

\section*{Acknowledgements}

The work of HM was partially supported by the JSPS grants: 
KAKENHI \#26K06861, \#25K07072, \#24K00531, \#26H02001. 
The work of PN was supported by the JSPS grant: KAKENHI \#26KF0103.
HVT is partially supported by NSF grant DMS-2348305.

\section*{Declarations}

\noindent {\bf Conflict of interest statement:} The authors state that there is no conflict of interest.

\medskip

\noindent {\bf Data availability statement:} Data sharing is not applicable to this article, as no datasets were generated or analysed during the current study.


\begin{thebibliography}{99}



\bibitem{AIPSM}
N.~Anantharaman, R.~Iturriaga, P.~Padilla, and H.~S\'anchez-Morgado,
\emph{Physical solutions of the Hamilton--Jacobi equation}, Discrete Contin. Dyn. Syst. Ser. B \textbf{5} (2005), no.~3, 513--528.


\bibitem{CGMT}
F.~Cagnetti, D.~Gomes, H.~Mitake, and H.~V.~Tran,
\emph{A new method for large time behavior of degenerate viscous
Hamilton--Jacobi equations with convex Hamiltonians},
Ann. Inst. H. Poincar\'e Anal. Non Lin\'eaire
\textbf{32} (2015), no.~1, 183--200.

\bibitem{CZ}
Q.~Chen and Z.-X.~Zhu,
\emph{A new selection problem for degenerate viscous Hamilton--Jacobi equations},
arXiv:2605.12996, 2026.

\bibitem{DFIZ}
A.~Davini, A.~Fathi, R.~Iturriaga, and M.~Zavidovique,
\emph{Convergence of the solutions of the discounted Hamilton--Jacobi equation},
Invent. Math. \textbf{206} (2016), no.~1, 29--55.

\bibitem{EvansAdjoint}
L.~C.~Evans,
\emph{Adjoint and compensated compactness methods for Hamilton--Jacobi PDE},
Arch. Ration. Mech. Anal. \textbf{197} (2010), no.~3, 1053--1088.

\bibitem{FS}
W.~H.~Fleming and H.~M.~Soner,
\emph{Controlled Markov Processes and Viscosity Solutions}, 2nd ed.,
Stochastic Modelling and Applied Probability, vol.~25, Springer, New York, 2006.


\bibitem{Gomes}
D.~A.~Gomes,
\emph{Generalized Mather problem and selection principles for viscosity
solutions and Mather measures},
Adv. Calc. Var. \textbf{1} (2008), no.~3, 291--307.


\bibitem{IMT}
H. Ishii, H. Mitake, H. V. Tran, 
\emph{The vanishing discount problem and viscosity Mather measures.  
Part 1: the problem on a torus}, 
 J. Math. Pures Appl. (9) 108 (2017), no. 2, 125--149.  


 \bibitem{ISM}
R.~Iturriaga and H.~S\'anchez-Morgado,
\emph{Limit of the infinite horizon discounted Hamilton--Jacobi equation},
Discrete Contin. Dyn. Syst. Ser. B \textbf{15} (2011), no.~3, 623--635.


\bibitem{JKM}
H. R. Jauslin, H. O. Kreiss, and J. Moser, 
\emph{On the forced Burgers equation with periodic boundary conditions}, Differential equations: La Pietra 1996 (Florence), Proc. Sympos. Pure Math., vol. 65, Amer. Math. Soc., Providence, RI, 1999, pp. 133–153.

 \bibitem{LMT}
 N. Q. Le, H. Mitake, H. V. Tran,
{Dynamical and Geometric Aspects of Hamilton-Jacobi and Linearized Monge-Amp\`ere Equations},
Lecture Notes in Mathematics 2183, Springer.

\bibitem{LTY}
Z.~Liu, H.~V. Tran, and Y.~Yu,
\emph{Nonexistence of vanishing-viscosity limits for mechanical Hamiltonian ergodic problems}, arXiv:2605.10478, 2026.

\bibitem{MitakeTran}
H.~Mitake and H.~V.~Tran,
\emph{Selection problems for a discount degenerate viscous Hamilton--Jacobi equation},
Adv. Math. \textbf{306} (2017), 684--703.


\bibitem{MitakeSoga}
H.~Mitake and K.~Soga,
\emph{Weak KAM theory for discounted Hamilton--Jacobi equations and its application},
Calc. Var. Partial Differential Equations \textbf{57} (2018), no.~3,
Art.~78, 32 pp.



\bibitem{NiRate}
P.~Ni,
\emph{Sharp convergence rates for the vanishing discount problem with hyperbolic Aubry sets},
arXiv:2609.02779, 2026.

\bibitem{NYZ}
P.~Ni, J.~Yan, and M.~Zavidovique,
\emph{Static class-guided selection of elementary solutions in non-monotone vanishing discount problems},
arXiv:2602.09697, 2026.


\bibitem{Hung-book}
H. V. Tran, Hamilton--Jacobi equations: theory and applications.
Graduate Studies in Mathematics, 213. American Mathematical Society.


\bibitem{TranAdjoint}
H.~V.~Tran,
\emph{Adjoint methods for static Hamilton--Jacobi equations},
Calc. Var. Partial Differential Equations \textbf{41} (2011), 301--319.


\bibitem{TuZhangEH}
S.~N.~T.~Tu and J.~Zhang,
\emph{On the regularity of stochastic effective Hamiltonian},
Proc. Amer. Math. Soc. \textbf{153} (2025), 1191--1203.

\bibitem{WZ}
Z.~Wang and J.~Zhang,
\emph{On the vanishing viscosity limit of Hamilton--Jacobi equations with nearly optimal discount}, SIAM J. Math. Anal. \textbf{58} (2026), no.~5, 4641--4669.


\bibitem{Yu}
Y.~Yu,
\emph{A remark on the semi-classical measure from
$-\frac{h^2}{2}\Delta+V$ with a degenerate potential $V$},
Proc. Amer. Math. Soc. \textbf{135} (2007), no.~5, 1449--1454.

\end{thebibliography}
\end{document}